\documentclass[a4papper]{article}

\usepackage[cp1251]{inputenc}
\usepackage[T2A]{fontenc}
\usepackage[russian,english]{babel}
\usepackage{amsmath, amssymb, amsthm, amsfonts}
\usepackage{titlesec}
\usepackage{graphicx}
\usepackage{indentfirst}
\usepackage{cmap}

\renewcommand\qedsymbol{$\blacksquare$} 

\titlelabel{\thetitle.\,}

\newtheorem {theorem}{Theorem}[section]

\newtheorem{corollary}{Corollary}[section]

\newtheorem{lemma}{Lemma}[section]

\theoremstyle{definition}
\newtheorem{example}{Example}[section]

\begin{document}

\title{On a Transmission Problem Related to Models of Electrocardiology}
\date{}
\author{Yulia L. Shefer\\ Institute of Mathematics, Siberian Federal University}

\maketitle

\abstract{
We a generalization of transmission problems for elliptic 
matrix operators related to the mathematical 
models of cardiology. We indicate sufficient conditions providing that 
the  approach elaborated for scalar elliptic operators is still valid 
in this much more general situation. 
}

\section*{Introduction}

In this paper we consider a family of transmission problems for elliptic operators with 
constant coefficients related to models of electrocardiology. More precisely, for many years 
for satisfactory models of heart activity one uses Cauchy, Dirichlet, and Neumann problems for 
scalar strongly elliptic operators, see, for example, \cite{1}, \cite{2}. A modification of 
such a model involving boundary problems for the Laplace operator has been recently studied 
in  \cite{3}. 

We consider similar problems for more general matrix linear elliptic operators and find 
sufficient conditions under which the scheme for solving the problems suggested in \cite{3} 
allows to construct their solutions. Our approach is essentially based on the general theory 
of Fredholm problems for strongly elliptic (matrix) linear operators, see, e.g., 
\cite{Simanca1987}, and the theory of regularization of an ill-posed Cauchy problem for 
operators with an injective principal symbol, see \cite{3}.

\section{A model example}
\label{sec:1}

To begin with, we consider a basic example related to models of electrocardiology. As known 
from clinical practice, see, e.g., \cite{1}, \cite{2}, electrical activity of cardiac cells is 
crucial for pumping function of heart, which is the result of rhythmical cycles of 
contraction-relaxation of the cardiac tissue. Anomalies of electrical activity often cause 
heart diseases, which makes these investigations, in particular, development of adequate 
mathematical models, very relevant nowadays.

Let us illustrate this by one model of electrocardiology  \cite{1,2,2K}. 
Denote by $\Omega_B$ and $\Omega_H$ three-dimensional domains with piecewise smooth boundaries 
with  $\partial \Omega_B$ and  $\partial \Omega_H$ corresponding to a body and a heart (see 
Fig. 1). Then the domain $\Omega = \Omega_B \setminus \Omega_H$ with the boundary $\partial 
\Omega = \partial \Omega_B  \cup \partial \Omega_H$ corresponds to the body without heart.

\begin{figure}[h!]
	\centering
		\includegraphics[scale=0.5]{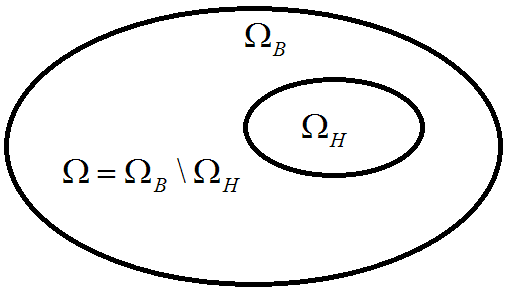}
	\label{fig:domain}
	\caption{Geometry of the model}
\end{figure}

Usually, in standard models one assumes that the cardiac tissue can be divided into two parts 
-- intracellular and extracellular parts separated by a membrane -- to which the electric 
potential $u_i$ and $u_e$, respectively, is assigned. Regarding the cardiac tissue as a 
continuous medium we think of the potentials as defined in each point of $\Omega_H$ and 
satisfying the equation
\begin{equation}
\nabla^*M_i \nabla u_i + \nabla^*M_e \nabla u_e = 0, 
\label{eq:1}
\end{equation}
where $M_i$ and $M_e$ are known tensor matrices that characterize intracellular and 
extracellular parts, and $\nabla$ is the gradient operator in $\mathbb{R}^3$.

One often considers the case when $M_i$ and $M_e$~are positively defined matrices with 
constant coefficients with entry values defined by conductivity of the cardiac tissue. For 
simplicity of the further analysis one assumes that these matrices are proportional
\begin{equation*}
M_i = \lambda M_e, \quad\ \lambda > 0.
\end{equation*}

Based on equation (\ref{eq:1}) one considers two models of heart activity. In one model it is 
assumed that the heart is isolated and one considers the problem
\begin{equation}
\begin{gathered}
	\nabla^*M_i \nabla u_i + \nabla^*M_e \nabla u_e = 0 \text{ в $\Omega_H$}, \\
	(\nu_1, \nu_2, \nu_3)M_i \nabla u_i = 0 \text{ на $\partial \Omega_H$}, \\
	(\nu_1, \nu_2, \nu_3)M_e \nabla u_e = -(\nu_1, \nu_2, \nu_3)M_b \nabla u_b \text{ на $
	\partial \Omega_H$}, \\
	u_b = u_e \text{ на $\partial \Omega_H$},
\end{gathered}
\label{eq:2}
\end{equation} 
where $M_b$~is the tensor matrix characterizing conductivity of the body, 
$\nu$ is the vector field of unit outward normal vectors to the boundary
of the domain under the consideration and  $u_{b}$ is the electric potential of the body.

In the second model one takes the body into account, and from the electrodynamics of 
stationary currents it follows that the electric potential of the body $u_b$ in the domain $
\Omega$ is defined by the equations

\begin{equation}
\begin{gathered} 
	\nabla^* M_b \nabla u_b = 0 \text{ in $\Omega$}, \\
	(\nu_1,  \nu_2, \nu_3)M_b \nabla u_b = 0 \text{ on $\partial \Omega_B$}. 
\end{gathered} 
\label{eq:3}
\end{equation}

A feature of the model is the fact that one is more interested not in potentials  $u_i$ and 
$u_e$ separately but in their difference $v=u_i-u_e$ in $\Omega_H$ or at least on its boundary. 

Since matrices $M_i$ and $M_e$ are positively defined and not degenerate, the problems 
\eqref{eq:2}, \eqref{eq:3} can be studied in the framework of the theory of boundary (maybe 
ill-posed) problems for elliptic formally self-adjoint equations, see \cite{1,2,2K}. Moreover, 
notice that the problems above may be regarded as transmission problems for elliptic equations 
with discontinuous coefficients describing solutions in different domains of a continuum with 
the help of additional conditions on separating surfaces, see, for example, \cite{Sche60}, 
\cite{Bor10}. 

Until now we have not used any functional spaces in the problems description, in the next 
section we give a precise formulation of a more general problem and specify functional classes 
for its solution.

\section{Formulation of a problem}
\label{sec:2}

Let $\theta$ be a measurable set in ${\mathbb R}^n$, $n\geq 2$.
Denote by $L^2(\theta)$ a Lebesgue space of complex-valued functions on  $\theta$ with the 
scalar product
\begin{equation*}
\left( u,v \right)_{L^2(\theta)} = \int_{\theta} \overline{v} (x)u (x)\ dx .
\end{equation*}
If $D$ is a domain in ${\mathbb R}^n$ with a piecewise smooth boundary
$\partial D$, 
then for $s \in \mathbb{N}$ we denote by $H^s(D)$ 
the standard Sobolev space with the scalar product
\begin{equation*}
\left( u,v \right)_{H^s(D)} = \int_{D}\sum_{|\alpha| \le s} 
\overline {(\partial^{\alpha}v)} (\partial^{\alpha}u)
dx .
\end{equation*}
It is well-known that this scale extends for all $s>0$. Let now $H^s(D)$ for 
$s\in {\mathbb R}_+\setminus {\mathbb Z}_+$ be the standard Sobolev-Slobodeckij spaces. 
Denote by $H^s_0 (D)$ the closure of the subspace $C^{\infty}_{\mathrm{comp}} (D)$ in 
$H^{s} (D)$, where $C^{\infty}_{\mathrm{comp}} (D)$ is the linear space of functions with 
compact supports in  $D$.

The space of $k$-vectors $u =(u_1, ...,  u_k)$ whose components lie in  
$H^s(D)$ equipped with the scalar product
\begin{equation*}
\left( u,v \right)_{[H^s(D)]^k} = \sum_{j=1}^k \int_{D}\sum_{|\alpha| \le s} 
\overline{(\partial^{\alpha}v_j)} (\partial^{\alpha}u_j)dx  = 
\int_{D}\sum_{|\alpha| \le s} 
(\partial^{\alpha}v)^* (\partial^{\alpha}u)dx 
\end{equation*}
we shall denote by $[H^s(D)]^k$.

Further on, we shall consider linear matrix operators  
\begin{equation*}
A= \sum_{|\alpha|\leq p} A_\alpha \partial ^\alpha, \, x \in D,
\end{equation*}
where $p\in \mathbb N$ is the order of operator $A$, $\alpha \in {\mathbb Z}^n_{+}$, and 
$A_\alpha$ are $(l\times k)$-matrices with constant coefficients.
By a formal adjoint of $A$ we call the differential operator 
\begin{equation*}
A^*= \sum_{|\alpha|\leq p} A^*_\alpha \partial ^\alpha,
\end{equation*}

\noindent
where $A^*_\alpha$ is the adjoint matrix for $A_\alpha$ or, equivalently, 
\begin{equation*}
	\left(  A u, v \right)_{[L^2(D)]^l} =  
	\left(  u, A^{*}  v \right)_{[L^2(D)]^k} 
	\,\ \text{для всех } u \in {\left[C_{0}^{\infty}(D)\right]}^k, 
	v \in {\left[C_{0}^{\infty}(D)\right]}^l.
\end{equation*}
As usual, the principal symbol of an operator  $A$ is the matrix
\begin{equation*}
\sigma(A)(x,\zeta)= \sum_{|\alpha|= p} A_\alpha  \zeta ^\alpha, \, x \in D, \, 
\zeta \in {\mathbb C}^n.
\end{equation*}
We say that the principal symbol of $A$ is injective if $l \geq k$ and
$$
	\text{rang} \,\sigma(A)(x,\zeta) = k, \,\ 
	\text{для всех } \zeta \in \mathbb{R}^{n} \setminus\left\{0\right\} \,\ 
	\text{ и всех } x \in \overline{D}.
$$
If $l=k$ operators with injective principal symbols are called elliptic.

Let now $A_e$, $A_i$, and $A_b$ be linear differential operators of the first order with
 constant coefficients on  $\overline{D}_m$, i.e.
\begin{equation*}
	A_m = \sum_{j=1}^{n} a_{j}^{(m)}\frac{\partial}{\partial x_j} + a_{0}^{(m)}, \\	
\end{equation*}
where $m \in \{e, i, b \}$,  $D_e \equiv D_i \equiv
 \Omega_H$, $D_b \equiv \Omega$. 

Further on, we assume that principal symbols of operators  $A_m$ are injective in the 
corresponding domains.

Denote by $A_{m}^{*}$ a formal adjoint of $A_{m}$ and consider a generalized Laplacian 
$A_{m}^{*} A_{m}$. 

Under assumptions made above, the operator $A_{m}^{*} A_{m}$ is a strongly elliptic 
$(k\times k)$-matrix second order operator, i.e. it is elliptic and there exists a positive 
constant  $c$ such that
$$
\Re\Big(-w^*\sigma(A_{m}^{*}A_{m})(x,\zeta)w \Big) 
\geq c\left |w\right|^2 |\zeta |^{2}\, \text{for all } \zeta \in \mathbb{R}^{n} 
\setminus\left\{0\right\}, \, w \in \mathbb{C}^{k} \setminus\left\{0\right\} \,,
 x \in \overline{D}_m. 
$$
The operator $A_{m}^{*} A_{m}$ is also formally self-adjoint, i.e.
\begin{equation*}
	\left( A_{m}^{*} A_{m} u, v \right)_{[L^2(D_m)]^k} =  
	\left(  u, A_{m}^{*} A_{m} v \right)_{[L^2(D_m)]^k} =
	\left(   A_{m} u, A_{m} v \right)_{[L^2(D_m)]^l} 	
	\,\ \text{for all } u,v \in {\left[C_{0}^{\infty}(D_m)\right]}^k;	
\end{equation*}
in particular, the operator  $A_{m}^{*} A_{m}$ is (formally) positively defined
\begin{equation*}
	\left( A_{m}^{*} A_{m} u, u \right)_{[L^2(D_m)]^k} \geq 0 \,\ \text{for all } u \in 
	{\left[C_{0}^{\infty}(D_m)\right]}^k.
\end{equation*}

Let, as before, $\nu$ be the outward normal vector operator on the boundary 
of the domain of the operator $A_{m}$.
Introduce the conormal derivatives
$$
	\nu_{A_{m}} = \sigma^{*}(A_m)(\nu)A_{m},
$$
associated with these operators via Green's formula:
\begin{equation} \label{eq.Green}
\int_{\partial\Omega} v \nu_{A_{m}} u ds = \int_{\Omega} (v^* (A_{m}^*A_{m} u) - (A_{m}v)^* 
A_{m}u)dx \text{   for all } u,v \in {\left[H^{2}(\overline{D_m})\right]}^k.	
\end{equation}

Assume that bounded domains $\Omega_H$, $\Omega$, and $\Omega_b$ have twice smooth 
boundaries and  consider the following problem \eqref{eq:4}-\eqref{eq:5}: 
find vector-functions $u_i$, $u_e$ from $ {\left[H^2(\Omega_H)\right]}^k$ and a vector-function 
 $u_b$ from ${\left[H^2(\Omega)\right]}^k$ such that 

\begin{equation}
\left\{
\begin{gathered}
	A_{i}^{*}A_i u_i + A_{e}^{*}A_e u_e = 0 \text{ in $\Omega_H$}, \\
	\nu_{A_i} u_i = 0 \text{ on $\partial \Omega_H$},  \\
	\nu_{A_e} u_e = - \nu_{A_b} u_b \text{ на $\partial \Omega_H$}, \\
	u_e = u_b \text{ на $\partial \Omega_H$}, 
\end{gathered}
\right.
\label{eq:4}
\end{equation}
\vspace{3mm}
\begin{equation}
\left\{
\begin{gathered}
	A_{b}^*A_b u_b = 0 \text{ in $\Omega$}, \\
	\nu_{A_b} u_b = 0 \text{ on $\partial \Omega_B$}, 
\end{gathered}
\right.
\label{eq:5}
\end{equation}
where the equality on the boundary is in the sense of traces, and the equality in the domains 
is in the sense of distributions. In this case we can assume that traces of functions and 
their conormal derivatives are well-defined.

It is obvious that the problem (\ref{eq:4}-\ref{eq:5}) is a generalization of the problem 
(\ref{eq:2}-\ref{eq:3}). Note also that it incorporates several classical boundary problems.

\begin{example} \label{ex.2.1} 
Consider first the classical case $A_{b}=\nabla$ ($k=1$, $l=n$),
then $\nu_{A_{b}} = \frac{\partial}{\partial\nu}$ is a directional derivative along the 
outward normal vector to $\partial \Omega_B$.
If we assume that $u_e$ is known on $\partial\Omega_H$ and equal to a function
$v_0 \in H^{3/2}(\partial \Omega_H)$, 
then (\ref{eq:4}-\ref{eq:5}) gives the following problem: 
find a function $u_b \in H^2(\Omega)$ satisfying
\begin{equation}
\left\{
\begin{gathered}
	-\Delta u_b = 0 \text{ in $\Omega$}, \\
	\frac{\partial u_b}{\partial \nu}  = 0 \text{ on $\partial \Omega_B$}, \\
	u_b = v_0 \text{ on $\partial \Omega_H$}.
\end{gathered}
\right.
\label{eq:5q}
\end{equation}
This is a classical mixed problem that is often called a Zaremba problem, see, e.g. 
\cite{Zare10}, \cite{Simanca1987}. This problem can be studied by standard methods in Sobolev 
and H\"older spaces. It is well-known that this problem has a unique solution in these classes 
that can be written with the help of the Green function ${\mathcal Z}_{\Omega}(x,y)$ having 
the standard properties
$$
	u_b (x)= \int\limits_{\partial\Omega}{\mathcal Z}
	_{\Omega}(x,y)v_0 (y)dS(y),\,\, x \in \Omega_H,
$$
where $dS(y)$ is the volume form on the surface $\partial\Omega$, 
see \cite{Zare10}, \cite{Simanca1987}. 

Analogously, if we assume that $A_{e}=\nabla$ 
($k=1$, $l=n$), then 
$\nu_{A_{e}} = \frac{\partial}{\partial\nu}$ is a directional derivative along the outward 
normal vector to $\partial \Omega_H$. If the conormal derivative $\nu_{A_e} u_e$ is known on 
$\partial\Omega_H$ and equal to a function $v_1 \in H^{1/2}(\partial \Omega_H)$, 
then \eqref{eq:4}-\eqref{eq:5} gives a special case of a classical Neumann problem for a 
Laplace operator: find a function $u_b \in H^2(\Omega)$ satisfying
\begin{equation}
\left\{
\begin{gathered}
	-\Delta u_b = 0 \text{ в $\Omega$}, \\
	\frac{\partial u_b}{\partial \nu}  = 0 \text{ on $\partial \Omega_B$}, \\
	\frac{\partial u_b}{\partial \nu} = v_1 \text{ on $\partial \Omega_H$},
\end{gathered}
\right.
\label{eq:5qq}
\end{equation}
see  \cite{Simanca1987}, \cite{Mikh}. It is known that this problem is Fredholm in Sobolev and 
H\"older spaces, its solution is defined up to an additive constant, and the necessary and 
sufficient condition for solvability is the following
\begin{equation} \label{eq.N.cond.solv}
 \int\limits_{\partial\Omega_H} v_1 (y)dS(y) =0.
\end{equation}
If this condition is satisfied the problem has a unique solution  $u_b$  in these classes that 
satisfies, for example,
\begin{equation} \label{eq.N.cond.uniq}
 \int\limits_{\partial\Omega_H} u_b (y)dS(y) =0.
\end{equation}
It can be written with the help of an appropriate parametrix ${\mathcal N}_{\Omega}(x,y)$ that 
has the standard properties
$$
	u_b (x)= \int\limits_{\partial\Omega}{\mathcal N}_{\Omega}(x,y)v_0 (y)dS(y),\,\, x \in 
	\Omega_H,
$$
However, the general theory of boundary problems suggests that knowledge of
$u_e$ or $\nu_{A_e} u_e$ on $\partial\Omega_H$ does not allow to recover the potential 
$u_i$ uniquely from the remaining data and equations without additional conditions (see also 
Uniqueness Theorem \ref{t.uniq} for the problem \eqref{eq:4}-\eqref{eq:5} proved under 
additional assumptions below). 

Besides that, cardiology models are special in the sense that additional conditions necessary 
for recovering of unknown potentials $u_i,\, u_e,\, u_b$ in the problem 
(\ref{eq:4})-(\ref{eq:5}) should preferably be set on the boundary of `the body' $\Omega$, 
since all measurements must be less traumatic for a patient and not invasive.
\end{example}

\section{Application of an ill-posed Cauchy problem}
\label{sec:3}

On of the simplest additional conditions mentioned above leads to using of an ill-posed Cauchy 
problem. More precisely, it implies measuring the potential $u_b$ on the boundary of `the 
body': 
\begin{equation}
u_b = f \text{ on $\partial \Omega_B$},
\label{eq:6}
\end{equation}
where $f$ is a given vector-function from ${\left[H^{3/2}(\Omega)\right]}^k$.

Unfortunately, as known very well, the problem \eqref{eq:5}, \eqref{eq:6} is nothing else but 
an ill-posed problem for an elliptic operator $A_{b}^*A_b$. Let us see what the addition of 
the property \eqref{eq:6} gives in a more general problem than those in cardiology.

Denote by $N(\Omega)$ the set of solutions to the problem 
\eqref{eq:4}, \eqref{eq:5}, \eqref{eq:6} under the condition $f=0$.
Let $S_{A_e}(\Omega_H)$ be the space of generalized solutions of the equation 
$A_{e}h = 0$ в $\Omega_H$. Since the operator $A_e$ has an injective symbol and its 
coefficients are real analytic, the Petrovsky theorem yields that the elements of the space
$S_{A_e}(\Omega_H)$ are real analytic vector-functions in $\Omega_H$.

\begin{theorem} \label{t.uniq}
Let bounded domains $\Omega_H$, $\Omega$, and $\Omega_b$ have twice smooth boundaries and let 
for some constant $\lambda>0$,
\begin{equation} \label{eq.proportional}
A_i = \lambda A_e .
\end{equation} 
Then the set $N(\Omega)$ consists of triples $(u_i,u_e,u_b)\subset 
{\left[H^2(\Omega_H)\right]}^k\times {\left[H^2(\Omega_H)\right]}^k\times {\left[H^2(\Omega)
\right]}^k$ such that
\begin{equation}
u_i=\frac{h-w}{{\lambda}^2},\ \ \  u_e=w,\ \ \  u_b=0,
\label{eq:7n}
\end{equation} 
where   $h$ is an arbitrary function from the space $S_{A_e}(\Omega_H) \cap 
{\left[H^2(\Omega_H)\right]}^k$, and $w$ is an arbitrary function from 
${\left[H^2_0(\Omega_H)\right]}^k$.
\end{theorem}

\begin{proof}

Let a vector  $h$ belong to $S_{A_e}(\Omega_H) \cap {\left[H^2(\Omega_H)\right]}^k$
and a vector  $w$ belong to $ {\left[H^2_0(\Omega_H)\right]}^k$. Then  
$w$ satisfies the following conditions
\begin{equation}
w=0 \text{ на $\partial\Omega_H$,} \,\, \nu_{A_i}(w) = 0 \text{ on $\partial\Omega_H$,}
\label{eq:8n}
\end{equation} 
and $A_i^* A_i = \lambda^2 A_e^* A_e $. 
Therefore the vector functions from (\ref{eq:7n}) give a solution to the problem  \eqref{eq:4}, 
\eqref{eq:5}, \eqref{eq:6} for $f=0$.

Let $u_i,u_e\in {\left[H^2(\Omega_H)\right]}^k$, and $u_b \in {\left[H^2(\Omega)\right]}^k$ is 
a triple of functions from $N(\Omega)$. Then from (\ref{eq:4})-(\ref{eq:5}) it follows that  
$u_b$ is a solution to the Cauchy problem for the operator   $A_{b}^*A_b $:
\begin{equation*}
A_{b}^*A_b u_b = 0 \text{ in $\Omega$},
\nu_{A_b}(u_b) = 0 \text{ on $\partial\Omega_B$},
u_b = 0 \text{ on $\partial\Omega_B$}.
\end{equation*} 
Since the operators $A_m$ have injective symbols, 
we have
\begin{equation*}
	\mbox{rang} \, (\nu_{A_{m}})(x,\nu (x)) = \sigma^{*}(A_m)(x,\nu (x)) 
	\sigma(A_{m}) (x,\nu(x)) =k
\end{equation*} 
for any $m=e,i,b$ and all $x\in \partial \Omega_H$ or $\partial \Omega$, 
respectively. In particular, the systems of boundary operators 
$\{ I, \nu_{A_{e}}\}$, $\{ I, \nu_{A_{i}}\}$ are first order Dirichlet systems on  
$\partial \Omega_H$, while the system of boundary operators $\{ I, \nu_{A_{b}}\}$  is a first 
order Dirichlet system on $\partial \Omega$ (see, for example, \cite{3}). Then by the 
uniqueness theorem for a Cauchy problem for elliptic operators   (see, for example, 
\cite[Theorem 10.3.5]{3}), $u_b \equiv 0$ in $\Omega$.
Now by the trace theorem for Sobolev spaces and by equations from (\ref{eq:4}) 
we see that  $u_e \equiv 0$ нon $\partial\Omega_H$ and $\nu_{A_e}(u_e) \equiv 0$ on 
$\partial\Omega_H$. However, since the system of boundary operators 
$\{ I, \nu_{A_{e}}\}$  is a first order Dirichlet system on $\partial \Omega_H$, it follows 
from the theorem on spectral synthesis (see 
\cite{HedbWolf1})  that $u_e \in [H^2_0(\Omega_H)]^k$. 

To complete the proof of the theorem we need the following lemma.

\begin{lemma}\label{lemma.h} 
Let $\Omega_H$ be a bounded domain in $\mathbb{R}^n$ with a twice smooth boundary and 
\eqref{eq.proportional}. If the functions $u_e$, $u_i \in {\left[H^2(\Omega_H)\right]}^k$ 
satisfy the equations (\ref{eq:4}) then they are related in $\Omega_H$ by
\begin{equation} \label{eq.h}
u_e + \lambda^2 u_i = h,
\end{equation} 
where $h$ as a function from the space $S_{A_{e}^{*}A_e}(\Omega_H) \cap 
{\left[H^2(\Omega_H)\right]}^k$.

Moreover, if $u_b \equiv 0$ on $\partial \Omega_H$, then the functions $u_e$, $u_i$ are 
related in $\Omega_H$ by \ref{eq.h}, where $h$ is a function from the space $S_{A_e}(\Omega_H)
 \cap {\left[H^2(\Omega_H)\right]}^k$.
\end{lemma}

\begin{proof}
\renewcommand\qedsymbol{$\square$}
Since $ A_i = \lambda A_e $, the first equation in (\ref{eq:4}) can be rewritten in the form 
\begin{equation}
{A_e}^*A_e h = 0 \text{ in $\Omega_H$},
\label{eq:9}
\end{equation} 
with $h = u_e + \lambda^2 u_i$, and clearly $h\in S_{A_{e}^{*}A_e}(\Omega_H) \cap 
{\left[H^2(\Omega_H)\right]}^k$.

If we additionally know that 
$u_b \equiv 0$ on $\partial \Omega_H$ then, as noticed above, 
$u_b \equiv 0$ in $ \Omega$. Therefore
$\nu_{A_e}(u_e)=0$ on $\partial\Omega_H$, and $\nu_{A_i}(u_i)=0$ and $\partial\Omega_H$, 
which implies that
\begin{equation}
\nu_{A_e}(h) = 0 \text{ on $\partial\Omega_H$}.
\label{eq:10}
\end{equation}
From this, by the Green formula \eqref{eq.Green} we obtain 
\begin{equation*}
\begin{gathered}
0 = (A_{e}^{*} A_e h, h)_{[L^2(\Omega_H)]^k} = \int \limits_{\Omega_H} h^* 
(A_{e}^{*} A_e h) dx = \\
= \int \limits_{\Omega_H} (A_e h)^* (A_e h) dx + \int \limits_{\partial\Omega_H} h^* 
\nu_{A_e}(h) ds = \left\| A_e h \right\|_{[L^2(\Omega_H)]^l}^2.
\end{gathered}
\end{equation*}

Therefore, the vector function  
$h$ defined by the equality \eqref{eq.h} belongs to 
$S_{A_e}(\Omega_H) \cap {\left[H^2(\Omega_H)\right]}^k$.
\end{proof}

Thus, the functions  $u_i,u_e\in {\left[H^2(\Omega_H)\right]}^k$ satisfy 
\eqref{eq:4}, and by Lemma \ref{lemma.h} we get $u_i=\frac{h-v}{{\lambda}^2}$, where $v\in 
[H^2_0(\Omega_H)]^k$  and $h\in S_{A_e}(\Omega_H)\cap {\left[H^2(\Omega_H)\right]}^k$.
\end{proof}

In particular, it follows from Lemma \ref{lemma.h} that the zero space of the problem  
\eqref{eq:4} coincides with the space 
$S_{A_{e}}(\Omega_H)\cap {\left[H^2(\Omega_H)\right]}^k$. 

Denote by $\ker A_e$ the kernel of a continuous linear operator 
$A_e: \, [H^2(\Omega_H)]^k\rightarrow [H^1(\Omega_H)]^l$ and consider several examples. 
In fact, $\ker A_e =S_{A_{e}}(\Omega_H)\cap {\left[H^2(\Omega_H)\right]}^k$. 

\begin{example}\label{ex.3.1}
 
Let $A_e = \begin{pmatrix} \nabla \\ 1 \end{pmatrix}$, ($k=1$, $l=n+1$).  
Then $A_{e}^{*} = \begin{pmatrix} -\text{div},  1 \end{pmatrix}$, 
$\nu_{A_{e}} = \frac{\partial}{\partial\nu}$, 
$A_{e}^{*} A_e = -\Delta + 1$, 
and the problem (\ref{eq:9})-(\ref{eq:10}) becomes a Neumann problem for the Helmholtz 
operator 
\begin{equation}
\label{eq:exam}
\left\{ 
\begin{aligned}
	-\Delta h + h = 0 \text{ in $\Omega_H$}, \\
	\frac{\partial h}{\partial \nu} = 0 \text{ on $\partial\Omega_H$},
\end{aligned}
\right.
\end{equation}
and the equation $A_e h = 0$ takes the form
\begin{equation*}
\left\{ 
\begin{aligned}
	\nabla h = 0 \text{ in $\Omega_H$}, \\
	h = 0 \text{ in $\Omega_H$}.
\end{aligned}
\right .
\end{equation*}
Consequently, $\text{ker} A_e =\{ 0 \}$ and coincides with the space of solutions of the 
homogeneous problem (\ref{eq:exam}).
\end{example}

\begin{example}\label{ex.3.2}

Let $A_e = \nabla$ then $A_e^* = - \text{div}$. 
In this case ($k=1$, $l=n$), 
$A_{e}^{*} = -\text{div}$, 
$\nu_{A_{e}} = \frac{\partial}{\partial\nu}$,   
$A_{e}^{*} A_e = - \Delta$, and the problem (\ref{eq:9})-(\ref{eq:10}) becomes a Neumann 
problem for the Laplace operator
\begin{equation}
\label{eq:exam1}
\left\{ 
\begin{aligned}	
	\Delta h = 0 \text{ in $\Omega_H$}, \\
	\frac{\partial h}{\partial \nu} = 0 \text{ on $\partial\Omega_H$},
	\end{aligned}
\right.
\end{equation}
and the equation $A_e h = 0$ takes the form
\begin{equation*}
	\nabla h = 0 \text{ in $\Omega_H$}.
\end{equation*}
Therefore, $\text{ker} A_e = \mathbb{R}$ and coincides with the space of solutions of the 
problem (\ref{eq:exam1}).

\end{example}

\begin{example}\label{ex.3.3}

Consider the case where $A_e = \overline{\partial} = \partial_x - i\partial_y$ is the Cauchy-
Riemann operator 
in ${\mathbb R^2}\cong \mathbb C$ where $i$ stands for imaginary unit.
Then $A_e^* = - \partial = -\partial_x - i\partial_y$, and the kernel of $A_e$ is holomorphic 
functions. The problem (\ref{eq:9})-(\ref{eq:10}) defines then the zero space of a non-
coercive $\overline \partial$-Neumann problem, see, for example,
\cite{FK}, \cite{ShlTark12}.

\end{example}

It is clear that the operator $A_e$ should be chosen in a way that its kernel is at least finite dimensional.

Under assumptions of Theorem \ref{t.uniq} the rest of the scheme of solving the problem 
\eqref{eq:4}, \eqref{eq:5}, \eqref{eq:6} differs little from the standard one, see \cite{2K}.
Namely, first we introduce a function  $h(x)$ such that $h(x)={\lambda}^2 u_i + u_e$, 
where $x\in \Omega_H$. 
From the conditions on the boundaries in (\ref{eq:4}) and the fact that $\nu_{A_i}=\lambda 
\nu_{A_e}$ we get that
\begin{equation*}
\nu_{A_e} h=-\nu_{A_b} u_b \text{ on $\partial\Omega_H$}.
\end{equation*}

Thus, we can rewrite the original problem \eqref{eq:4}, \eqref{eq:5}, \eqref{eq:6} in new 
notation: knowing a vector $f \in [H^{3/2}(\partial\Omega_H)]^k$, find vectors 
$h \in [H^{2} (\partial \Omega_H)]^k$ and 
$u_b \in [H^{2} (\partial \Omega)]^k$ such that
\begin{equation}
\left\{
\begin{gathered}
	A_{e}^{*}A_e h = 0 \text{ in $\Omega_H$},\\
	\nu_{A_e} h=-\nu_{A_b} u_b \text{ on $\partial\Omega_H$},\\
	\end{gathered}
\right.
\label{eq:NeumannProblem}
\end{equation}

\begin{equation}
\left\{
\begin{gathered}
A_{b}^{*}A_b u_b = 0 \text{ in $\Omega$},\\
\nu_{A_b} u_b = 0 \text{ on $\partial\Omega_B$},\\
u_b= f  \text{ on $\partial\Omega_B$}.
\end{gathered}
\right.
\label{eq:CauchyProblem}
\end{equation}

The original problem splits into two --  \eqref{eq:NeumannProblem} and 
\eqref{eq:CauchyProblem}. 
The problem \eqref{eq:CauchyProblem}, as noticed above, is an ill-posed Cauchy problem for an 
elliptic operator $A_b^*A_b$. It is known that if a solution to this problem exists it is 
unique.  
The problem \eqref{eq:NeumannProblem} is a Neumann problem for an elliptic operator 
$A_e^*A_e$. Unfortunately, in general the Neumann problem may also be ill-posed. For it to be 
Fredholm, the so called Shapiro-Lopatinsky conditions must be placed \cite[Chapter 1, \S 3, 
condition II for $q=0$]{Agr}, \cite{Shap} on the pair 
$(A_e^*A_e, \nu_{A_e} )$. 
In particular, they guarantee that the space $S_{A_e}(\Omega_H)\cap [H^2 (\Omega_H)]^k$ is 
finite dimensional.

More precisely, let us consider the following Neumann problem: for a given vector  
$h_0 \in [H^{1/2} (\partial \Omega_H)]^k$ find 
a vector $h \in [H^{2} (\partial \Omega_H)]^k$ such that 
\begin{equation}
\left\{
\begin{gathered}
	A_{e}^{*}A_e h = 0 \text{ in $\Omega_H$},\\
	\nu_{A_e} h= h_0 \text{ on $\partial\Omega_H$},\\
	\end{gathered}
\right.
\label{eq:NeumannProblem1}
\end{equation} 
and formulate conditions for solvability.

\begin{theorem} \label{t.Neumann}
If for a pair of operators $(A_e^*A_e, \nu_{A_e} )$ the Shapiro-Lopatinsky conditions are 
fulfilled then the problem \eqref{eq:NeumannProblem1} is Fredholm. To be precise, 
\begin{enumerate}
\item[1)] the zero space of the problem coincides with the finite-dimensional space $S_{A_e}(
\Omega_H)\cap [H^2 (\Omega_H)]^k$;
\item[2)] 
 the problem is solvable if and only if  
\begin{equation}\label{condition}
(h_0,\varphi)_{[L^2 (\partial\Omega_H)]^k}=0 
\mbox{ for all } \varphi \in S_{A_e}(\Omega_H)\cap [H^2 (\Omega_H)]^k;
\end{equation} 
\item[3)]
under \eqref{condition} there exists a unique solution $h_1$ 
of the problem \eqref{eq:NeumannProblem1} satisfying
\begin{equation}\label{eq.h1}
(h_1,\varphi)_{[L^2 (\partial\Omega_H)]^k}=0 
\mbox{ for all } \varphi \in S_{A_e}(\Omega_H)\cap [H^2 (\Omega_H)]^k.
\end{equation} 
\end{enumerate}
\end{theorem}

\begin{proof}
See \cite{Simanca1987}. 
\end{proof}

Thus, under hypothesis of Theorem \ref{t.Neumann} for solvability of the Neumann problem 
\eqref{eq:NeumannProblem} it is necessary and sufficient that for the vector 
$h_0=-\nu_{A_b} u_b$ the condition \eqref{condition} is fulfilled.
This can be achieved if we place additional conditions on relations between the operators 
$A_e$ and $A_b$. 
Namely, as we have seen above, it is quite natural to assume that
\begin{equation} \label{eq.prop.be}
A_e=\tilde{\lambda}A_b 
\mbox{ for some constant } \tilde \lambda>0.
\end{equation}

Denote by $S_{A_b}(\Omega)$ the zero space of solutions to the problem 
(\ref{eq:CauchyProblem}) in the domain $\Omega$.

\begin{corollary} \label{c.Neumann}
Let for the pair of operators $(A_e^*A_e, \nu_{A_e} )$ the Shapiro-Lopatinsky conditions be 
fulfilled. Besides that assume that the identity \eqref{eq.prop.be} holds and the spaces  
$S_{A_e}(\Omega_H)\cap [H^2 (\Omega_H)]^k$ and $S_{A_b}(\Omega)\cap [H^2 (\Omega)]^k$ 
coincide. Then for any vector $u_b \in [H^2 (\Omega)]^k$ satisfying \eqref{eq:CauchyProblem} 
there exists a unique vector  $h_1 \in [H^2 (\Omega_H)]^k$ that satisfies  
\eqref{eq:NeumannProblem} and \eqref{eq.h1}. 
\end{corollary}
\begin{proof} By Theorem \ref{t.Neumann} for solvability of the problem 
\eqref{eq:NeumannProblem} it is necessary and sufficient  that
\begin{equation}
(\nu_{A_e} u_b,\varphi)_{[L^2 (\partial\Omega_H)]^k}=0 
\mbox{ for all } \varphi \in S_{A_e}(\Omega_H)\cap [H^2 (\Omega_H)]^k.
\label{cor1}
\end{equation}
If the vector $u_b \in [H^2 (\Omega)]^k$ satisfies \eqref{eq:CauchyProblem}, 
then by the Green formula 
 \eqref{eq.Green} for the operator $A_b$ 
\begin{equation*}
- \int\limits_{\partial\Omega_H} \nu_{A_b} u_b \psi ds 
=
\int\limits_{\partial\Omega}\nu_{A_b} u_b \psi ds = (\psi ,A^*_b A_b u)_{[L^2 (\Omega)]^k} 
- (A_b \psi , A_b u)_{[L^2 (\Omega)]^k} =
 0,
\end{equation*}
for any $\psi \in S_{A_b}(\Omega_H)\cap [H^2 (\Omega_H)]^k$.

On the other hand, the relation \eqref{eq.prop.be} 
guarantees that $(\tilde \lambda)^2\nu_{A_b} = -\nu_{A_e} $,
and therefore 
\begin{equation*}
(\nu_{A_e} u_b,\varphi)_{[L^2 (\partial\Omega_H)]^k} = -
(\tilde \lambda)^2(\nu_{A_b} u_b,\varphi)_{[L^2 (\partial\Omega_H)]^k} = -
(\tilde \lambda)^2 \int\limits_{\partial\Omega_H} \nu_{A_b} u_b \psi ds
\end{equation*}
for any $\varphi \in S_{A_e}(\Omega_H)\cap [H^2 (\Omega_H)]^k$. 
Due to the fact that the spaces $S_{A_e}(\Omega_H)\cap [H^2 (\Omega_H)]^k$ and $S_{A_b}(\Omega)
\cap [H^2 (\Omega)]^k$ coincide, \eqref{cor1} holds. Then by statement 3 of Theorem 
\ref{t.Neumann} for any vector $u_b$  there exists a unique vector $h_1 \in [H^2 (\Omega_H)]^k$
 satisfying \eqref{eq:NeumannProblem} and \eqref{eq.h1}.
\end{proof}

The condition that the spaces $S_{A_e}(\Omega_H)\cap [H^2 (\Omega_H)]^k$ and 
$S_{A_b}(\Omega)\cap [H^2 (\Omega)]^k$ coincide seems to be rather strong, especially since 
these are spaces of solutions to different differential equations in different domains. 
Nevertheless, provided \eqref{eq.prop.be} holds,
such a coincidence is possible if  the operator $A_e$ is so much overdetermined that the space 
of its solutions in any domain is finite dimensional and coincides with the space of solutions 
in ${\mathbb R}^n$; the typical examples are the so-called 
stationary holonomic systems.
Let us illustrate this by the following examples.

\begin{example}
Let $A_e=\nabla$ and $A_b=\tilde{\lambda}\nabla$ ($k=1$, $l=n$). The function $u=const$ is a 
solution to the equation $\nabla u=0$ in $\Omega_H$ and extends to $\Omega$, where it is a 
solution to $\tilde{\lambda}\nabla{u}=0$. Thus we get that the spaces $S_{A_e}(\Omega_H)$ and 
$S_{A_b}(\Omega)$ coincide.
\end{example}

\begin{example}
Let $A_e = \begin{pmatrix} \nabla \\ 1 \end{pmatrix}$ and $A_b= \tilde{\lambda}\begin{pmatrix}
 \nabla \\ 1 \end{pmatrix}$, ($k=1$, $l=n+1$).  
A solution to $A_e u = 0$ in $\Omega_H$ is $u\equiv{0}$ and it extends to $\Omega$, where it 
is a solution  $A_b u=0$. Thus, the spaces $S_{A_e}(\Omega_H)$ and
$S_{A_b}(\Omega)$ coincide.
\end{example}

\begin{example}
Consider the following operators $A_i$, $A_e$ и $A_b$:
$$
A_e = 
\begin{pmatrix}
\partial_x & 0 & 0\\
\partial_y & 0 & 0\\
0 & \partial_x & 0\\
0 & \partial_y & 0\\
-1 & 0 & \partial_x\\
0 & -1 & \partial_y\\
\end{pmatrix}, \,\
A_i = \lambda
\begin{pmatrix}
\partial_x & 0 & 0\\
\partial_y & 0 & 0\\
0 & \partial_x & 0\\
0 & \partial_y & 0\\
-1 & 0 & \partial_x\\
0 & -1 & \partial_y\\
\end{pmatrix}, \,\
A_b = \tilde{\lambda}
\begin{pmatrix}
\partial_x & 0 & 0\\
\partial_y & 0 & 0\\
0 & \partial_x & 0\\
0 & \partial_y & 0\\
-1 & 0 & \partial_x\\
0 & -1 & \partial_y\\
\end{pmatrix}.
$$
These operators have injective principal symbols and are equivalent to second order operators
$$
\tilde{ A_e } = \begin{pmatrix} \partial_{xx} \\ \partial_{yy} \\ \partial_{xy} \end{pmatrix}, 
\,\,
\tilde{ A_i } = \lambda \begin{pmatrix} \partial_{xx} \\ \partial_{yy} \\ \partial_{xy} 
\end{pmatrix}, 
\,\,
\tilde{A_b}= \tilde{\lambda} \begin{pmatrix} \partial_{xx} \\ \partial_{yy} \\ 
\partial_{xy},  \end{pmatrix}.
$$ 
Therefore the space of solutions of the system $A_e u = 0$ in 
$\Omega_H$ coincides with the set of all linear functions $u=c_1x+c_2y+c_3$,  
and any function of this form extends to $\Omega$, where it is a solution to the equation
$A_b u=0$. Therefore, the spaces $S_{A_e}(\Omega_H)$ and 
$S_{A_b}(\Omega)$ coincide.
\end{example}

As noticed above, if a solution to the Neumann problem \eqref{eq:NeumannProblem} exists, it 
is `unique' up to an element of the space $S_{A_e}(\Omega_H)\cap [H^2 (\Omega_H)]^k$ (
additively). 

Recall that the aim of solving the original problem  \eqref{eq:4}, \eqref{eq:5}, \eqref{eq:6} 
is to find the transmembrane potential $v$ on the surface $\partial\Omega_H$. Let us write 
down the algorithm for solving  the problem \eqref{eq:NeumannProblem}, 
\eqref{eq:CauchyProblem}:

\begin{enumerate}
	\item find a function $u_b $ and its conormal derivative $\nu_{A_b} (u_b) $ on the surface  
	$\partial\Omega_H$ by solving an ill-posed Cauchy problem \eqref{eq:CauchyProblem} 	for an 
	elliptic operator $A_b^*A_b$.	
	\item compute values of $h(x)$ on the surface $\partial\Omega_H$ by solving a Neumann 
	problem \eqref{eq:NeumannProblem} for an elliptic operator $A_e^*A_e$ 
with the data $\nu_{A_b} u_b$ on $\partial\Omega_H$ obtained in Step 1. The possibility of 
this depends on whether the restrictions on operators $A_e$ and $A_b$ described above hold. 
	\item find the transmembrane potential $v$ on the surface $\partial\Omega_H$ by using the 
	relation \eqref{eq.h} together with $u_b$ and $h$ on $\partial\Omega_H$, found in Steps 1 
	and 2, respectively
\begin{equation}
v=u_i-u_e=\frac{h-u_b}{{\lambda}^2} - u_b  \text{ on $\partial\Omega_H$}.
\label{eq:13}
\end{equation}
\end{enumerate}

In conclusion we note that solvability conditions for an  ill-posed Cauchy problem in Sobolev 
spaces for a rather wide class of operators with real analytic coefficients are well known, 
see, for example, \cite{3}. Moreover, in \cite{3}, \cite{4} on can find constructive 
procedures for its regularization, i.e. for construction exact and approximate solutions (the 
so called Carleman formulas). 
Regarding areas with models with the geometry corresponding to that in cardiology and 
operators that are first order matrix factorizations of the Laplace operator, or more 
generally, of a Lam\'e-type operator such Carleman formulas were obtained in \cite{Sh97}.

\medskip

\emph{ The work was supported by the Foundation for the
Advancement of Theoretical Physics and Mathematics “BASIS”.  }


\end{document}